\documentclass[a4paper,reqno,oneside]{article}

\usepackage[OT1]{fontenc}
\usepackage[english]{babel}
\usepackage[utf8]{inputenc}
\usepackage{amsmath}
\usepackage{amssymb}
\usepackage{amsthm}
\usepackage{dsfont}

\usepackage{a4wide} 

\usepackage{algorithm}
\usepackage{algorithmic}

\newtheorem{thm}{Theorem}[section]
\newtheorem{lem}[thm]{Lemma}
\newtheorem{cor}[thm]{Corollary}

\newcommand\calC{\ensuremath{\mathcal{C}}}
\newcommand{\N}{{\ensuremath \mathbb{N}}}

\begin{document}
\begin{center}

\LARGE The game chromatic number of dense random graphs
\vspace{8mm}

\Large{
\begin{tabular}{ccc}
Ralph Keusch & \quad & Angelika Steger \\
{\small{rkeusch@inf.ethz.ch}} & \quad & {\small{steger@inf.ethz.ch}}
\end{tabular}
}
\vspace{5mm}

\large
  Institute of Theoretical Computer Science \\
  ETH Zurich, 8092 Zurich, Switzerland
\vspace{8mm}

\small

\begin{minipage}{0.9\linewidth}
\textsc{Abstract.}
Suppose that two players take turns coloring the vertices of a given graph $G$ with $k$ colors. In each move the current player colors a vertex such that neighboring vertices get different colors. The first player wins this game if and only if at the end, all the vertices are colored. The \textit{game chromatic number $\chi_g(G)$} is defined as the smallest $k$ for which the first player has a winning strategy.

Recently, Bohman, Frieze and Sudakov [Random Structures and Algorithms 2008] analysed the game chromatic number of random graphs and obtained lower and upper bounds of the same order of magnitude. In this paper we improve existing results and show that with high probability, the game chromatic number $\chi_g(G_{n,p})$ of dense random graphs is asymptotically twice as large as the ordinary chromatic number $\chi(G_{n,p})$.
\end{minipage}

\vspace{5mm}
\end{center}


\section{Introduction}\label{sec:intro}

Consider the following Maker-Breaker game played on a graph~$G$ whose vertices are uncolored at the beginning. During the game Maker and Breaker alternately take turns and color one vertex per move such that the coloring remains proper, i.e.,\ two neighbors never receive the same color. Maker's goal is to ensure that all vertices get colored, while Breaker's aim is to avoid this by reaching a partial coloring that cannot be extended any more. Maker has the first move.

The {\em game chromatic number} $\chi_g(G)$ is defined as the smallest number of colors for which Maker has a winning strategy, no matter how Breaker plays. Obviously, $\chi_g(G)$ is at least as large as the chromatic number~$\chi(G)$. On the other side, Maker always wins the game if the number of colors is larger than the maximum degree of $G$, because then no vertex can run out of colors. Therefore the parameter $\chi_g(G)$ is well-defined. Usually, $\chi_g(G)$ is larger than $\chi(G)$. 

In fact, the difference between the two parameters can be very large. Consider for example the complete bipartite graph $B_{n,n}$ minus a perfect matching~$M$. The chromatic number of this graph is two, while Breaker has a winning strategy whenever the number of colors is less than~$n$: if Maker colors some vertex~$v$ with color~$i$ then Breaker uses the same color on the vertex~$w$ which is matched to~$v$ in the matching~$M$. Color~$i$ can henceforth not be used on any other vertex, and the claim follows by induction.

In this paper, we study the game chromatic number of the Erdős–Rényi random graph model $G_{n,p} = (V,E)$. 
We assume throughout the paper that $p=p(n) \le 1 - \eta$, where $\eta > 0$ is an arbitrarily small, fixed constant. Bohman, Frieze and Sudakov \cite{suda} determined upper and lower bounds for $\chi_g(G_{n,p})$ for a wide range of edge probabilities $p$. Let $b = \frac{1}{1-p}$ and note that for $p=o(1)$ we have $\log_b np = \frac{\log np}{\log b}=(1+o(1))\frac{\log np}{p}$. (We use $\log x$ to denote the logarithm to base $e=2.71..$.) 


\begin{thm}[Theorem 1.1 in \cite{suda}]\label{thm:bfs}\hfill

\begin{list}{}{\leftmargin0.6cm\labelwidth0.4cm\labelsep0.2cm\topsep0pt\itemsep0pt}
\item[(a)\hfill] There exists a constant $K > 0$ such that for $\varepsilon > 0$ and $p \geq (\log n)^{K\varepsilon ^{-3}}/n$, it holds with high probability
\[\chi_g(G_{n,p}) \geq (1-\varepsilon)\frac{n}{\log_b{np}}.\]
\item[(b)\hfill] If $\alpha > 2$ is a constant, $K=\max\{\frac{2\alpha}{\alpha-1},\frac{\alpha}{\alpha-2}\}$ and $p \geq (\log n)^K/n$, then with high probability
\[\chi_g(G_{n,p}) \leq \alpha \frac{n}{\log_b{np}}.\]
\end{list}
\end{thm}

Shortly after, Frieze, Haber and Lavrov obtained almost as good estimates for sparse random graphs \cite{neues}. Note that for the ordinary chromatic number it is well-known by the results of Bollobás and Łuczak  \cite{bollo, lucz} that with high probability, $\chi(G_{n,p}) = (1+o(1))\frac{n}{2\log_b{np}}$. Hence on random graphs, the two parameters have the same order of magnitude. 

In this paper, we improve the upper bound for dense random graphs.

\begin{thm}\label{thm:main}
Let $p \ge e^{-o(\log n)}$ and $b := \frac{1}{1-p}$. Then with high probability,
\[\chi_g(G_{n,p}) \leq (1+o(1))\frac{n}{\log_b{np}}.\]
\end{thm}

In particular, this result holds for constant values of $p$. Together with the first statement of Theorem~\ref{thm:bfs} it implies the asymptotic value of the game chromatic number for dense random graphs:

\begin{cor}
Let $p \ge e^{-o(\log n)}$ and $b := \frac{1}{1-p}$. Then with high probability,
\[\chi_g(G_{n,p}) = (1+o(1))\frac{n}{\log_b{np}} =  (2+o(1))\chi(G_{n,p}).\]
\end{cor}

Note that for any constant $\xi > 0$, our lower bound $p \ge e^{-o(\log n)}$ implies $(np)^\xi p \gg \log np \cdot \log^2 n$ and thus
\begin{equation}\label{eq:boundonp}
(np)^\xi \gg \log_b np \cdot \log^2 n,
\end{equation}
which we will use several times in the proof of Theorem~\ref{thm:main}. In Section 2 we give an overview of our proof strategy and describe the most important concepts. Afterwards Section 3 contains the main part of the technical work and finishes the proof of Theorem~\ref{thm:main}.

\section{Outline of Proof Strategy}

Suppose that $p \ge e^{-o(\log n)}$. We need to show that for any constant $\alpha > 1$, arbitrarily close to one (but not equal), and a number of colors $k=\alpha\frac{n}{\log_b np}$ Maker has a strategy so that he wins the game with probability $1-o(1)$. For the rest of the paper we assume that~$p$ and~$\alpha$ are fixed and~$k$ is chosen as above.

By $\mathcal{C} = (\calC_1, \calC_2, \ldots, \calC_k)$ we denote a collection of pairwise disjoint sets, where $\calC_i$ is the set of all vertices which have been assigned with color~$i$. Note that we do not require that the sets~$\calC_i$ partition the vertex set. In this way we can view $\calC$ as the partial coloring obtained after some given number of vertices have been colored.

For a vertex~$v$ we denote by $A(v,\mathcal{C})$ the set of all colors which are still available at~$v$ with respect to the partial coloring~${\cal C}$. That is,
\[A(v,\mathcal{C}) = \{i \in \{1, \ldots, k\}: N(v) \cap \calC_i = \emptyset \},\]
where $N(v)$ denotes the neighborhood of $v$. Furthermore, define
\[a(v,\mathcal{C}) = \left|A(v,\mathcal{C}) \right|.\]

During the evolution of the game, more and more vertices get colored and the sets $A(v,\calC)$ will shrink. Maker needs to avoid that a set $A(v,\mathcal{C})$ of an uncolored vertex gets empty. This indicates that the coloring game bears some relation to the so-called \textit{box game} introduced by Chvatal and Erd\H{o}s~\cite{erdoes}, cf.\ also Hamidoune and Las~Vergnas~\cite{corr} for some corrections and generalizations. In this box game, Maker and Breaker take turns in claiming previously unclaimed elements of some given, pairwise disjoint sets. Maker wants to claim at least one element from each set, while Breaker's goal is to prevent this, thus Breaker wants to claim all elements of at least one set. (Note that in this description we have deliberately changed the roles of Maker and Breaker in comparison to the original setting, as this better fits our purposes.)
A natural strategy for Maker is to play greedily, i.e., to always claim an element from the set that currently has the smallest number of elements. If Maker claims an element from a set, we can remove this set from the game as it is not dangerous any more. Denote by $B(A_1, \ldots, A_k)$ the box game on $k$ pairwise disjoint sets $A_1, \ldots, A_k$ where Maker has the first move. If some conditions on the sets $A_i$ are fulfilled, the greedy strategy allows Maker to win the game.  Before presenting details, we introduce some generalizations to the box game so that we can apply it for the analysis of the game chromatic number. 

If Maker colors a vertex~$v$ in the coloring game, then this color has to be removed from the sets $A(w,\mathcal{C})$ for all neighbors $w\in N(v)$. The coloring of $v$ corresponds to a move of Maker in the box game, while the removals correspond to a move of Breaker. Assume we know that each color appears in at most $q+1$ sets. Then we could enhance the power of Breaker by allowing him to remove from at most $q$ sets $A(w,\calC)$ an arbitrary color. With that Breaker has more power, but if we can show that Maker wins this generalized game, then he will also win the original coloring game. 

In the coloring game, Breaker colors vertices as well in his turns, which should be translated to a Maker-move in the box game. We model this as follows: we allow Breaker to steal every other move from Maker. That is, instead of Maker choosing an element (from the smallest set) we allow Breaker to claim an element from an arbitrary set, label it as a Maker element, remove the set (that contained the element) from the game, and then proceed with his own move by eliminating at most $q$ elements from the remaining sets. Note that this essentially means that after one (real) Maker move Breaker may remove one set plus $2q$ elements from the remaining sets.

If we allow Breaker to steal one of Maker's moves, we might as well allow him to steal more than one. We denote by $B(A_1, \ldots, A_k;q,z)$ the box game on~$k$ pairwise disjoint sets $A_1, \ldots, A_k$ where Maker has the first move and
Breaker claims at most~$q$ elements per move and is allowed to steal all but every $z$-th of  Maker's moves. 
In the next lemma we formulate a criterion when Maker can win a game $B(A_1, \ldots, A_k;q,z)$. Thereby, we also allow Maker to play only $d$-greedily, meaning that he always chooses a set that contains at most~$d$ elements more than the currently smallest available set.

\begin{lem}\label{l:boxgame}
Let $q,d,z \in \mathds{N}$ and let $f(1,q,d) := zq+d$ and 
\[f(k,q,d) := (zq+d)k\left(1+\sum_{i=1}^{k-1}\frac{1}{i}\right)\qquad\text{for $k > 1$.}\]
If $\sum_{i \in I} |A_i| > f(|I|,q,d)$ holds for all non-empty subsets $I \subseteq \{1, \ldots, k\}$, then Maker wins the game $B(A_1, \ldots, A_k; q,z)$ by playing $d$-greedily.
\end{lem}
\begin{proof}
We prove this statement by induction, always looking at periods of~$z$ moves of Maker, the first of these moves being a real Maker's move, the remaining $z-1$ stolen by Breaker. Playing $d$-greedily, Maker claims in his real moves always an element from a set that contains at most $d$ more elements than the currently smallest set. 

For the induction assumption we show that Maker cannot lose in the first period. 
The first period consists of~$z$ Maker moves, the last $z-1$ of which stolen by Breaker. In this period Breaker can claim at most $zq$ elements from the $k-z$ sets remaining after the~$z$ Maker moves (real or stolen). By definition of the function $f$ we know that $|A_i| > zq$ holds for all~$i$, therefore no set $A_i$ runs out of elements during the first moves of Breaker, and Maker doesn't lose in the first period of the game. 

If $k \leq z$, this proves the statement. So assume $k > z$ and let $\tilde{I}$ be the remaining index set when Maker plays a real move the next time. Then  $|\tilde{I}| = k-z$. For $i \in \tilde{I}$ we denote by $A'_i$ the remaining set after the first~$z$ moves of Maker/Breaker. Let $I' \subseteq \tilde{I}$ be any non-empty subset of size $\ell = |I'| > 0$, and let $I = I' \cup \{j\}$, where $A_j$ denotes the set which Maker claimed in his first move. Recall that we assumed that Maker claims in his first move an element from a set that contains at most $d$ elements more than the currently smallest set. We thus deduce that for $t := \sum_{i \in I} |A_i|$, $t^{\ast} := \sum_{i \in I'} |A'_i|$ satisfies
\begin{align*}
t^{\ast} & \;\geq\; t - |A_j| - zq \;\geq\; t - \lfloor t/(\ell+1)\rfloor  -d - zq \;\geq\; \frac{\ell}{\ell+1} t - d-zq\\
& \;>\; \frac{\ell}{\ell+1} f(\ell+1,q,d)-d-zq \;=\; (zq+d)\ell\left(1+\sum_{i=1}^\ell \frac{1}{i}\right) - d-zq\\
& \;=\; (zq+d)\ell\left(1+\sum_{i=1}^{\ell-1} \frac{1}{i}\right) \;=\; f(|I'|,q,d).
\end{align*}
Together with the induction hypothesis this shows that Maker will not lose the game in the later part of the game.
\end{proof}

We are now ready to define Maker's strategy in the coloring game. One of Maker's goals is to ensure that the color classes grow almost uniformly. Clearly, Maker cannot achieve this completely, as Breaker can play arbitrarily. But at least he can make sure that no color class is too small. Let $N$ be a constant chosen appropriately later. Then Maker's strategy is the following:


\begin{list}{$\bullet$}{\leftmargin0.7cm\labelsep0.2cm}
\item In every $N$-th move, Maker chooses an uncolored vertex $v$ such that $a(v,\mathcal{C})$ is minimal, where $\mathcal{C}$ is given by the current color classes, and assigns any color $i \in A(v,\mathcal{C})$ to $v$. We call this  a \textit{move of first type}.
\item In all his other turns, Maker chooses a color $i$ such that $|\calC_i|$ is minimal among all colors that can still be used somewhere and assigns $i$ to an uncolored vertex $v$ with $i \in A(v,\mathcal{C})$. We call this a \textit{move of second type}.
\end{list}

Note that we may assume that Maker's strategy is deterministic: we fix some arbitrary ordering on the vertices, so that we can break ties uniquely. With these preliminaries at hand we are now ready to outline the main idea of our proof strategy. 

Assume Breaker wins at time~$t$, i.e., assume that after $t-1$ vertices have been colored, we have a vertex $v_0$ that is still uncolored and for which all $k$ colors appear in the neighborhood $N(v_0)$. To reach a contradiction we will then define a time $t'<t$ and argue that we may view the coloring game between times $t'$ and $t$ as a box game, where the sets $A_i$ correspond to the sets $A(v,\calC)$ for the vertices that were colored between times $t'$ and $t$ plus the vertex $v_0$ (that ran out of colors). To see that this box game is a Maker's win (and that therefore the coloring game could not have stopped at time $t$), we will argue that the conditions of Lemma~\ref{l:boxgame} are satisfied. For this we need that the sets $A(v,\calC)$ are large {\em in comparison} to the power of Breaker. Recall that the power of Breaker (the  parameter $q$ in Lemma~\ref{l:boxgame}) corresponds to how often a color appears in the sets $A(v,\calC)$, maximized over all colors. Clearly, the larger $t$ (and thus $t'$), the smaller $q$ and the sets $A(v,\calC)$.  To carefully balance these effects we partition the game into phases, parametrized by a parameter~$h$. In each phase we will use different bounds for the size of the sets $A(v,\calC)$ and the power of Breaker.  Set $b=1/(1-p)$ as in Theorem~\ref{thm:main} and let  $\xi>0$ be a constant which we will define later. We define three functions as follows:
\begin{align}\label{eq:def:fcts}
\beta(h) \;&:=\; \frac{\alpha \xi n (np)^{-h}}{10\log_b np} = \Theta\left(\frac{(np)^{1-h}}{\log np}\right),\nonumber\\
\gamma(h)\;&:=\; \frac{10n\log n}{\beta(h)} = \Theta\left(\frac{(np)^h}{p}\log n \log np\right)\quad\text{and}\\
q(h) \;&:=\; \frac{\beta(h)}{(\log{n})^2} = \Theta\left(\frac{(np)^{1-h}}{(\log n)^2 \log np}\right)\nonumber.
\end{align}

With these definitions at hand we can now show that the conditions of Lemma~\ref{l:boxgame} are satisfied under various assumptions.

\begin{lem}\label{l:applyboxgame}
Let $\alpha,p,k,N,\beta(\cdot)$ and $q(\cdot)$ be as defined %
above and assume that Maker plays according to our proposed strategy. Let $t'\le n$ be a point in time, let $U$ be a set of uncolored vertices at time $t'$ and denote by $\calC'$ the coloring after the first $t'-1$ vertices have been colored. Furthermore, assume that there exists a constant $h<1$ such that the following conditions are satisfied:
\begin{list}{}{\leftmargin0.8cm\labelsep 0.1cm\labelwidth 0.7cm\topsep5pt\itemsep0pt}
\item[$(i)$\hfill]at time $t'$ Maker colors some vertex by using a move of his first type,
\item[$(ii)$\hfill]$a(v,\calC')\ge \beta(h)/2$ for all $v\in U$, 
\item[$(iii)$\hfill]for all $v\in U$ there exists $S(v)\subseteq A(v,\calC')$ s.t.\ $|S(v)|\le q(h)$ and s.t.\
$$|\{ v\in U\mid i\in A(v,\calC')\setminus S(v)\}| \le q(h)\qquad\text{for all colors $i=1,\ldots,k$}.$$
\end{list}
Then Maker does not lose the coloring game in the interval $[t',t'+|U|]$, if within this time interval both players color only vertices of $U$.
\end{lem}
\begin{proof}We have already seen the connection between the coloring game and the box game. We will show that Maker wins the box game induced by the vertex set $U$ and the color sets $A(v,\calC ') \setminus S(v)$. 

Every color appears at most $q(h)$ times in the sets $A(v,\calC ') \setminus S(v)$, which implies a Breaker-power of at most $q(h)$ in our box game translation. Recall that in his moves of first type, Maker chooses a vertex~$v$ where $a(v,\mathcal{C})$ is minimal. Since $|S(v)|\le q(h)$ for all $v \in U $, we know that on the sets $A(v,\calC ') \setminus S(v)$, Maker plays $q(h)$-greedily with his moves of first type. Note that if one player colors a vertex $v$ using a color $i$, we remove anyway the whole box $A(v,\calC')$ from the box game and don't care if $i \in S(v)$. This allows us to look at the box game of the restricted sets $A(v,\calC ') \setminus S(v)$.

Maker uses his move of first type in all his $N$-th turns. In the box game this corresponds to the setting where Breaker steals all but every $2N$-th of Maker's moves. We conclude that between $t'$ and $t$, Breaker and Maker have played the box game 
\[B(A(v_1,\mathcal{C}')\setminus S(v_1), A(v_2,\mathcal{C}')\setminus S(v_2), \ldots, A(v_{|U|},\mathcal{C}')\setminus S(v_{|U|}); q(h),2N).\] 
We observe that 
\begin{align*}
(2N q(h)+q(h))\left(1+\sum_{i=1}^{|U|-1}\frac{1}{i}\right) & = \Theta(q(h) \log|U|) \;=\;
o(\beta(h))
\end{align*}
by definition of $q(h)$ and $\log|U|\le \log n$. For $n$ large enough we thus have for all $v \in U$ that 
\[|A(v,\mathcal{C}') \setminus S(v)| \geq \frac{\beta(h)}{2}-q(h) \gg (2N q(h)+q(h))\left(1+\sum_{i=1}^{|U|-1}\frac{1}{i}\right).\]

By Lemma~\ref{l:boxgame} Maker wins this box game, thus no vertex of the set $U$ could run out of available colors until time $t'+|U|$. Hence Maker does not lose the coloring game in this period.
\end{proof}

Lemma~\ref{l:applyboxgame} shows that it is not essential that we can bound how often colors appear at uncolored vertices. It suffices if  we can put those colors that appear too often (and therefore enlarge Breaker's power) in some sets $S(v)$. The critical and most technical part of our proof will be to show that we can find such sets $S(v)$ in order to apply Lemma~\ref{l:applyboxgame}. In the remainder of this section we give an outline of the key steps.

First we study how the sizes of our color classes behave during the coloring process. For doing so we introduce some notation. We call a color~$i$ {\em active} if there exists at least one uncolored vertex that has $i$ in its color set $A(v,\calC)$. The {\em level} of the game, given a partial coloring $\calC$, is then defined as the minimum size of an active color class:
\begin{equation}\label{eq:def:ell}
\ell(\calC) := \min\{ |\calC_i| \mid \text{ $i$ active}\}. 
\end{equation}

Similarly, we define $\ell(t) = \ell(\calC)$, where $\calC$ is the partial coloring obtained at time $t$. Due to Maker's moves of second type, $\ell(t)$ is increasing during the game. In Section~\ref{sec:midlevel} we show that if Maker uses our proposed strategy, then his moves of the second type imply that as long as we have enough uncolored vertices, there are always many active colors classes whose size is close to the current level of the game. 

\begin{lem}\label{l:midlevel}
Let $\alpha,p$, and $k$ be as in Theorem~\ref{thm:main} and define $\xi = \frac{1}{10}(1-\frac{1}{\alpha})$. Then with high probability, there exists a constant $N=N(\alpha)$ such that the following statement is true for all $t \leq n- (np)^{1-4\xi}$:
If Maker plays according to our proposed strategy with parameter $N$, then the total number of active colors $i$ with size $|C_i| \leq \ell(t)+\xi\log_b np$ is at least $\frac{\xi}{8}k $.\end{lem}

Recall that Maker's strategy is deterministic. The event in the above lemma thus depends solely on properties of the random graph. Moreover note that since $\alpha > 1$, $\xi$ is a small but positive constant.

In Section~\ref{sec:color} we prove the following lemma.

\begin{lem}\label{l:color}
Let $\alpha,p,k,\xi$ and $\gamma(\cdot)$ be defined as above. Let $t' \le n$ be a point in time, $U$ be a set of uncolored vertices and $4\xi < h \le \frac{1}{\alpha}+3\xi$ be a constant such that $|C_i| \ge (h-4\xi)\log_b np$ holds for every active color $i$ at time $t'$ and such that $|U| \le 2N\gamma(h)+1$. If conditions $(i)$ and $(ii)$ of Lemma~\ref{l:applyboxgame} are satisfied and Maker uses the strategy defined above with parameter $N$ given by Lemma~\ref{l:midlevel}, then with high probability condition $(iii)$ is satisfied as well.
\end{lem}

Basically the statement of this lemma is that if we can control the quantities of the game process with a single constant $h$, then we can apply Lemma~\ref{l:applyboxgame} in order to show that Maker wins the coloring game. It is important to note that the high probability statements holds only for a fixed constant~$h$. In order to complete the proof of Theorem~\ref{thm:main} we do, however, need to apply these lemmas for different values of~$h$. We achieve this by dividing the game process into a constant number of periods, which are defined via the level $\ell(t)$ of the game. Using Lemma~\ref{l:midlevel} we will then show that for each period it suffices to consider a single constant $h$. Thus we need to apply Lemma~\ref{l:color} only a constant number of times, which is fine. Section~\ref{sec:proofofthm} contains the details of these arguments and how we can use them to complete the proof of Theorem~\ref{thm:main}.

\section{Proofs}

\subsection{Properties of random graphs}\label{sec:generalprop}

We start this proof section by collecting some properties of partial colorings~$\calC$ of a random graph $G_{n,p}$.  Thereby we will repeatedly use the following Chernoff-type bounds for the tails of the binomial distribution:

\begin{thm}[cf.\ e.g.\ \cite{alonspencer}]\label{thm:chernoff}Let $X = \sum_{i=1}^n$ be the sum of independent indicator random variables such that $\Pr[X_i = 1] = p_i$. Then the following inequalities hold for $\mu := \mathds{E}[X] = \sum_{i=1}^n p_i$:
\begin{list}{}{\leftmargin0.8cm\labelsep 0.1cm\labelwidth 0.7cm\topsep5pt\itemsep0pt}
\item[$(i)$\hfill]$\Pr[X \leq (1-\delta)\mu] \leq e^{-\mu \delta^2/2}$\quad for all $0 < \delta \leq 1$, and
\item[$(ii)$\hfill]$\Pr[X \ge t] \leq 2^{-t}$\quad for all $t \geq 2e\mu$.
\end{list}
\end{thm}

Recall that a color~$i$ is called {\em active}, if it occurs in at least one set $A(v,\calC)$ of an uncolored vertex~$v$. We say that a color is {\em eliminated} if it is not active any more. Clearly, a color~$i$ may be eliminated if $|\calC_i|$ is very large. For instance Breaker could use the same color over and over again until it is nowhere possible. It turns out that with high probability, as long as we have enough uncolored vertices, the total number of colors which are eliminated before they are heavily used is relatively small. The following lemma formalizes this. (Think of~$A$ as a set of uncolored vertices where none of the colors $1,\ldots,d$ can be used.)

\begin{lem}\label{l:eliminated}
Let $e^{-o(\log n)} \le p \le 1-\eta$,  $b=\frac{1}{1-p}$, $\alpha>1$ be a constant and $\xi = \frac{1}{10}(1-\frac{1}{\alpha})$. Put $d:=\lceil (np)^{1-4\xi}\rceil$. Denote by $\mathcal{A}$ the event that there are pairwise disjoint sets $A$, $\calC_1$, $\calC_2$, \ldots, $\calC_d$ in $G_{n,p}$ such that 
\begin{list}{$\bullet$\hfill}{\labelwidth0.4cm\leftmargin0.5cm\labelsep0.1cm\topsep0pt\itemsep0pt}
\item $|A| \geq d$,
\item $|C_i| \leq (\frac{1}{\alpha}+\xi)\log_b{np}$ for all $i \in \{1, \ldots, d\}$, and
\item for all $i \in \{1, \ldots, d\}$ and $v \in A$ we have $\calC_i \cap N(v)\not=\emptyset$.
\end{list}
Then
$$\Pr[\mathcal{A}]=o(1).$$
\end{lem}
\begin{proof}
Choose a set $A$ with $|A| = d$ and sets $\calC_i$ with $|\calC_i|=(\frac{1}{\alpha}+\xi)\log_b{np}$. (Note that we ignore floors and ceilings for ease of notation.) For a fixed vertex $v \in A$ we have
\[
\Pr\left[\bigwedge_{i=1}^d \left\{ N(v) \cap \calC_i \neq \emptyset \right\} \right]  \;=\; (1-(1-p)^{|\calC_i|})^d 
 \;=\; (1-(np)^{-1/\alpha-\xi})^d  \; \leq \; e^{-(np)^{-1/\alpha - \xi}d},
\]
and a union bound implies 
\begin{align*}
\Pr[\mathcal{A}] \;&\leq\; \binom{n}{|A|} \binom{n}{(\frac{1}{\alpha}+\xi)\log_b np}^d  e^{- |A|(np)^{-1/\alpha-\xi }d} \\
& \leq\; n^{|A|+(\frac{1}{\alpha}+\xi)d \log_b{np}}  e^{- |A|(np)^{-1/\alpha -\xi}d} \\
& =\; n^{(np)^{1-4\xi}+(\frac{1}{\alpha}+\xi)(np)^{1-4\xi} \log_b{np}} e^{-(np)^{2-9\xi-1/\alpha}}\\
& =\; o(1),
\end{align*}
as  $2-9\xi-1/\alpha=1+\xi$ and $p$ satisfies $(\ref{eq:boundonp})$.
\end{proof}

A direct consequence of Lemma~\ref{l:eliminated} is that we can bound the level $\ell(t)$ of the game (cf.\ definition~$(\ref{eq:def:ell})$) from above, assuming that there are enough uncolored vertices. 

\begin{cor}\label{l:levelbound}
Let $e^{-o(\log n)} \le p \le 1-\eta$,  $b=\frac{1}{1-p}$, $\alpha>1$ and $\xi = \frac{1}{10}(1-\frac{1}{\alpha})$. Let $t \leq n-(np)^{1-4\xi}$. Then with high probability, $\ell(t) < (\frac{1}{\alpha}+\xi)\log_b np$.
\end{cor}
\begin{proof}Suppose by contradiction that $\ell(t) \geq (\frac{1}{\alpha}+\xi)\log_b np$. By the assumption $t \leq n-(np)^{1-4\xi}$ we know that there are at least $(np)^{1-4\xi}$ uncolored vertices. From Lemma~\ref{l:eliminated} it follows that with high probability, at most $d=\lceil (np)^{1-4\xi}\rceil$ colors have been eliminated before they reached size at least $(\frac{1}{\alpha}+\xi)\log_b{np}$. All other color classes have size at least $(\frac{1}{\alpha}+\xi)\log_b np$ at time~$t$. But this would imply immediately that the total number of colored vertices is at least
\[\left(\frac{\alpha n}{\log_b np}-\lceil (np)^{1-4\xi}\rceil\right)\left(\frac{1}{\alpha}+\xi\right)\log_b np \;=\; (1-o(1))(1+\alpha\xi) n \;>\; n,\]
which is not possible.
\end{proof}

One of the crucial points in the proof of Theorem~\ref{thm:main} is to find sets $S(v)$ that satisfy Condition $(iii)$ of Lemma~\ref{l:applyboxgame}. Intuitively, this is easier if for a given partial coloring $\calC$, the color lists $A(v,\calC)$ look almost randomly, i.e., if they contain different colors for different vertices $v$. Our next lemma establishes some bounds on how similar the sets $A(v,\calC)$ can be. 


\begin{lem}\label{l:shuffle}Let $p$, $\xi$, $q(\cdot)$, $k$ be defined as above and let $4\xi < h \le \frac{1}{\alpha}+3\xi$ and $0< c < 1$ be constants. Let $\mathcal{B}$ denote the event that there exist disjoint sets $S$, $\calC_1, \ldots, \calC_d$ in $G_{n,p}$, where $d =d(n)\le k$ with  $|\calC_i| \geq (h-4\xi)\log_b{np}$ for all $i \in \{1, \ldots d\}$ and $|S| \leq (np)^{h+\xi}$, such that one of the following two conditions are satisfied:
\begin{list}{}{\topsep0pt\itemsep0pt\labelwidth0.4cm\labelsep0.2cm\leftmargin0.6cm}
\item[$(i)$\hfill]$d \leq   c q(h) |S| (np)^{-6\xi}$ and for all $v \in S$ there exist at least $ c q(h)$ sets $\calC_i$ such that $N(v) \cap \calC_i = \emptyset$, or
\item[$(ii)$\hfill]$d \ge |S|q(h)^{-1}(np)^{6\xi}$ and for all $i \in \{1, \ldots, d\}$ there exist at least $q(h)$ vertices $v \in S$ such that $N(v) \cap \calC_i = \emptyset$.
\end{list}
Then 
\[\Pr[\mathcal{B}] = o(1).\]
\end{lem}

\begin{proof}
Note that without loss of generality we may assume that the $\calC_i$'s all have size {\em equal} to $C := (h-4\xi)\log_b{np}$. We can then
apply a union bound over the choices of $d$, $S$ and $\calC_1, \ldots, \calC_d$ to observe that
\begin{align*}
\Pr[{\cal B}] &\le \sum_{d,s}\binom{n}{s}\cdot \binom{n}{C}^d\cdot \Pr[{\cal B}(d,S,\calC_1,\ldots,\calC_d)],
\end{align*}
where $\Pr[{\cal B}(d,S,\calC_1,\ldots,\calC_d)]$ denotes the probability that ${\cal B}$ holds for some fixed sets $S$ and $\calC_1,\ldots,\calC_d$, and $s$ denotes the size of the set $S$.
Observe that for a vertex $v \in S$ and a given set $\calC_i$ we have
\begin{equation}\label{eq:shuffle}
\Pr[N(v) \cap \calC_i = \emptyset] = (1-p)^{|\calC_i|} = (np)^{-h+4\xi},
\end{equation}
as we assumed that $|\calC_i| =  C$.

We first consider property $(i)$. Clearly, in this case we may assume that $d \leq  c q(h) |S| (np)^{-6\xi}$. For every $v \in S$ we define a random variable $X(v)$ that counts the number of sets $\calC_i \in \{\calC_1, \ldots, \calC_d\}$ such that $v$ has no neighbors in $\calC_i$.  By $(\ref{eq:shuffle})$ and the upper bounds on $d$ and $|S|$ we deduce 
\[
\mathds{E}[X(v)]  \leq d \cdot (np)^{-h+4\xi} = |S|  \cdot c q(h) (np)^{-6\xi} (np)^{-h+4\xi} \leq  c q(h) (np)^{-\xi} \ll  c q(h).
\]
$X(v)$ is a sum of independent Bernoulli random variables. By a Chernoff bound (see Theorem~\ref{thm:chernoff}) it follows that $\Pr[X(v) \geq  c q(h)] \leq 2^{- c q(h)}$.
Note that the random variables $X(v)$ are independent for all $v\in S$. Hence, we have for $n$ sufficiently large
\[
\binom{n}{C}^d\ \cdot \Pr[\text{$(i)$ holds}]
 \;\leq\; n^{dC} \cdot 2^{- c q(h) |S|}
 \;=\; n^{(h-4\xi)\log_bnp \cdot  c q(h) |S| (np)^{-6\xi} } \cdot 2^{- c q(h) |S|}
 \;\stackrel{(\ref{eq:boundonp})}\le\; 2^{-\frac12 c q(h) |S|} .
\]

Now consider property $(ii)$. Here we may assume w.l.o.g.\ that 
$d = |S|q(h)^{-1}(np)^{6\xi}$. For every $i \in \{1,\ldots,d\}$ define a random variable $X(i)$ that counts the number of vertices in $S$ that have no neighbor in $\calC_i$. By $(\ref{eq:shuffle})$, the upper bounds on $|S|$ and $h$ and the definition of $\xi=\frac1{10}(1-\frac1{\alpha})$ this implies
\[\mathds{E}[X(i)] \leq |S| \cdot (np)^{-h+4\xi} \leq (np)^{5\xi} \leq  (np)^{1-h-2\xi} \ll q(h).\]
Similarly as above we use Chernoff bounds to obtain
$\Pr[X(i) \geq q(h)] \leq 2^{-q(h)}$
and the independence of the random variables $X(i)$ to deduce that for $n$ sufficiently large
\[
\binom{n}{C}^d\ \cdot \Pr[\text{$(ii)$ holds}]
  \;\leq\; n^{dC} \cdot 2^{-q(h) d}
 \;=\; n^{(h-4\xi)\log_b{np} \cdot |S|q(h)^{-1}(np)^{6\xi}} \cdot 2^{- |S|(np)^{6\xi}}
 \;\stackrel{(\ref{eq:boundonp})}\le\; 2^{- \frac12|S|(np)^{6\xi}}.
\]
Combining both cases we conclude
\begin{align*}
\Pr[\mathcal{B}] & \;\le\; \sum_{d,s} n^s \left(2^{-\frac12 c q(h) s} + 2^{- \frac12s (np)^{6\xi}}\right) \;=\;o(1),
\end{align*}
as claimed.
\end{proof}

Given a partial coloring $\calC$ of $G_{n,p}$, a vertex $v$ is dangerous if the set $A(v,\calC)$ is small. These sets $A(v,\calC)$ shrink during the game process, while the level increases. Our last lemma of this section puts the total number of small sets $A(v,\calC)$ in relation with the level $\ell(\calC)$ and shows that with high probability, every partial coloring has the property that there are not many dangerous vertices with respect to $\ell(\calC)$.

\begin{lem}\label{l:sudakov}Let $p, \xi, k, \beta(\cdot), \gamma(\cdot)$ be defined as above and let $h < 1$ be a constant. For all partial colorings $\calC$, define
\[B(h,\calC) := \{v \in V: a(v,\calC) < \beta(h)/2\}.\]
Denote by $\mathcal{E}$ the event that there exists a partial coloring $\calC$ of the graph $G_{n,p}$ such that 
\begin{list}{$\bullet$\hfill}{\labelwidth0.4cm\leftmargin0.5cm\labelsep0.1cm\topsep0pt\itemsep0pt}
\item $\ell(\calC) \le (h - \xi) \log_b np$,
\item $|B(h,\calC)|\ge \gamma(h)$, and
\item $|C_i| \le \ell(\calC) + \xi \log_b np$ holds for at least $\frac{\xi}{10}k$ color classes $\calC_i$.
\end{list}
Then 
\[\Pr[\mathcal{E}]=o(1).\]
\end{lem}

\begin{proof}Let $h < 1$ and let $\calC$ be a partial coloring of $G_{n,p}$ such that $\ell(\calC) \le (h - \xi) \log_b np$. Denote by $I$ the set of all color classes $C_i$ which satisfy $|C_i| \le \ell(\calC) + \xi \log_b np \le h \log_b np$. Assume that $|I| \geq \frac{\xi}{10}k$. Then for any vertex $v \in V$, it holds
\begin{equation}\label{ineq:exp}
\mathds{E}[a(v,\mathcal{C})]  \;=\; \sum_{i=1}^k (1-p)^{|\calC_i|} \;\geq\; \sum_{i \in I} (1-p)^{|\calC_i|} 
  \;\geq\; \frac{\xi k}{10} (1-p)^{h \log_b np} \;=\;  \frac{\alpha \xi n (np)^{-h}}{10\log_b np} \;=\; \beta(h).
\end{equation}

Note that the number of colors available at a fixed vertex $v$ is the sum of independent indicator variables $X_i$, where $X_i = 1$ if and only if $v$ has no neighbors in $C_i$. Chernoff bounds thus imply
\[\Pr[a(v,\calC) < \beta(h)/2] \;\leq\; e^{-\beta(h)/8},\]
and therefore
\[\Pr[|B(h, \calC)| \ge \gamma(h)] \;\leq\; \binom{n}{\gamma(h)} e^{-\beta(h)\gamma(h)/8} = \binom{n}{\gamma(h)} n^{-\frac{5}{4}n}.\]
There are $(k+1)^n$ different partial colorings of the graph. We finish the argument by applying a union bound over all partial colorings that satisfy the assumption of the lemma. This yields
\[\Pr[\mathcal{E}] \leq (k+1)^{n} \binom{n}{\gamma(h)} n^{-\frac{5}{4}n} \leq n^{n+o(n)-\frac{5}{4}n} = o(1).\]
\end{proof}

\subsection{Proof of Lemma~\ref{l:midlevel}}\label{sec:midlevel}

Before we consider the proof of Lemma~\ref{l:midlevel} we study a balls-and-bins game.
Suppose we have~$k$ bins and two players~$M$ and~$B$ who alternately put a ball into one of the bins. We don't know how~$B$ plays, but~$M$ chooses always the bin with minimum load. But we have the following two exceptions: $B$ steals every~$N$-th ball of $M$ and plays this ball himself, and~$B$ can remove bins at any point in time.

In this model we will use $t$ and $\ell(t)$ in a similar way as defined in the setting of the coloring game. That is,~$t$ denotes the time (number of balls played) and $\ell(t)$ denotes the number of balls in the bin with minimum number of balls at time~$t$. In addition, denote by $t(\ell)$ the minimal time $t$ such that $\ell(t)=\ell$.

\begin{lem}\label{l:ballgame}
Consider the ball-game described above with~$k$ bins and parameter~$N$. Let $a \in \N$ and denote by $C(\ell)$ the total number of balls which have been thrown at loads $\ell'$, $\ell < \ell' \leq a$, until time $t(\ell)$. Then it holds for all $\ell < a$ that
\[C(\ell) \leq \frac{k \ell(N+1)(a-\ell)}{(N-1)(a-1)}.\]
\end{lem}

\begin{proof}We prove this Lemma by induction over~$\ell$. For $\ell=0$, clearly $C(0)=0$ which agrees with the formula. Let $\ell < a-1$ and suppose the statement is true for $\ell$. By definition, all non-removed bins have load at least~$\ell$ at time $t(\ell)$. Denote by~$x$ the number of bins that have load exactly~$\ell$ at time $t(\ell)$. Using the definition of $C(\ell)$ we observe that $x \leq k - \frac{C(\ell)}{a-\ell}$.

$M$ chooses always minimum-loaded bins, therefore~$M$ can throw at most $x$ balls between $t(\ell)$ and $t(\ell+1)$. Note that if~$B$ removes some bins in this time period, then the upper-bound on $M$-balls is even smaller. Since $B$ steals every $N$-th ball of~$M$, $B$ can play at most $\frac{N+1}{N-1}x$ balls between $t(\ell)$ and $t(\ell+1)$.

Recall that $C(\ell+1)$ counts the number of balls which have been thrown at loads between $\ell+1$ and~$a$ until time $t(\ell+1)$. We may count all $\frac{N+1}{N-1}x$ $B$-balls, but the balls at load $\ell+1$ at time $t(\ell)$ don't count for $C(\ell+1)$.  It follows that
\begin{align*}
C(\ell+1) & \;\leq\; C(\ell) + \frac{N+1}{N-1}x - (k-x) \;=\; C(\ell) - k + \frac{2N}{N-1}x \\
& \;\leq\; C(\ell) - k + \left(k-\frac{C(\ell)}{a-\ell}\right)\frac{2N}{N-1}\\
& \;\leq\; k \frac{N+1}{N-1} + C(\ell) \left(1-\frac{2N}{(a-\ell)(N-1)}\right)\\
& \underset{\text{ind.}}{\;\leq\;} k\frac{N+1}{N-1}+\frac{k \ell (N+1)(a-\ell)}{(N-1)(a-1)}\cdot \frac{a-\ell-2}{a-\ell}\\
& \;=\; k\frac{N+1}{N-1}\left(1+\frac{\ell(a-\ell-2)}{a-1}\right) 
 \;=\; k\frac{N+1}{N-1}\frac{(a-\ell-1)(\ell+1)}{a-1}.
\end{align*}
\end{proof}

A direct consequence of Lemma~\ref{l:ballgame} is that there exists always a constant fraction of bins whose load is close to the actual level of the process.

\begin{cor}\label{cor:ball}
Let $0 < \xi < 1$ and let $a$ be an integer large enough such that $\frac{(1-\xi)a+1}{a-1} \leq 1 - \frac{\xi}{2}$. Consider the ball-game described above with~$k$ bins and parameter $N \geq \frac{8}{\xi}$, where we suppose that~$B$ removes at most $\frac{\xi}{8}k$ bins with load less than~$a$.  Let $t$ be a point in time such that $\ell(t') \leq a(1-\xi)$. Then there exist at least $\frac{\xi}{8}k$ non-removed bins which have load at most $a$ at time $t$.
\end{cor}
\begin{proof}
Let $\ell' := \ell(t)+1$ and suppose that after time~$t$, the two players continue with the process until $t' := t(\ell')$, that is, until the game reaches a new level. Then $C(\ell') \leq \frac{k \ell'(N+1)(a-\ell')}{(N-1)(a-1)}$ holds by Lemma~\ref{l:ballgame}. There are at most $\frac{C(\ell')}{a-\ell'}$ bins which have load at least~$a$ at time~$t'$. However, since the bin-loads are increasing, this property holds also at time~$t$. Taking into account also the removed bins, we obtain that the total number of non-removed bins which have load at most $a$ at time~$t$ is at least
\begin{align*}
k-\frac{C(\ell')}{a-\ell'}-\frac{\xi}{8}k & \;\geq\; k\left(1-\frac{\ell'(N+1)}{(N-1)(a-1)}-\frac{\xi}{8}\right)\\ & \;\geq\; k\left(1-\frac{N+1}{N-1}\cdot\frac{(1-\xi)a+1}{a-1}-\frac{\xi}{8}\right)\\
& \;\geq\; k\left(1-\frac{\xi}{8}-\frac{N+1}{N-1}\cdot\left(1-\frac{\xi}{2}\right)\right),
\end{align*}
where we used $\ell' \leq a(1-\xi)+1$ and our choice of~$a$. It remains to show that $\frac{N+1}{N-1}(1-\frac{\xi}{2}) \leq 1-\frac{\xi}{4}$. But this inequality is guaranteed by $N \geq \frac{8}{\xi}$.
\end{proof}
The connection between this model and the coloring game is straightforward: the bins correspond to the colors, and if a ball falls into a bin this means that some vertex has been assigned with this color. Maker's moves of second type are equal to~$M$'s strategy of playing the balls. In order to perform a worst-case analysis, we donate Maker's moves of first type to player~$B$ in the ball game. At last, if a color is eliminated, we model this by a removed bin. This allows us to prove Lemma~\ref{l:midlevel} using the ball model.

\begin{proof}[Proof of Lemma~\ref{l:midlevel}]
Assume that the random graph does not satisfy event~$\mathcal{A}$ from Lemma~\ref{l:eliminated}, which happens with probability $1-o(1)$. Then Corollary~\ref{l:levelbound} and the definition of $\mathcal{A}$ imply for all $t \le n-(np)^{1-4\xi}$ that we have $\ell(t) \leq (\frac{1}{\alpha}+\xi)\log_b np$ and that at most $\lceil (np)^{1-4\xi}\rceil \ll \frac{\xi}{8}k$ colors have been eliminated before they reached size $\ell(t)$.
Let $a := \ell(t) + \xi\log_b np$. Then
\[\frac{\ell(t)}{a} \;=\; \frac{\ell(t)}{\ell(t)+\xi\log_b np} \;\leq\; \frac{1/\alpha + \xi}{1/\alpha + 2\xi} \;\leq\;  1-\xi,\]
which ensures $\ell(t) \leq a(1-\xi)$. In addition we have for $n$ large enough that
\[\frac{(1-\xi)a+1}{a-1} \;=\; (1-\xi)(1+o(1)) \;\leq\; 1 - \frac{\xi}{2}.\]
If we thus set $N=n(\alpha)$ to an arbitrary integer of size at least $8/\xi$ then the assumptions of Corollary~\ref{cor:ball} are all satisfied and the lemma follows.
\end{proof}

\subsection{Proof of Lemma~\ref{l:color}}\label{sec:color}

Define
\begin{equation}\label{eq:def:lc}
L := \frac{h+2\xi}{2\xi} \quad \text{and}\quad c :=  c(h) =  \frac{1}{L+1}.
\end{equation}
Note that both $L$ and $c$ are constants depending only on $\alpha$ and $h$. In the following we assume that the random graph does not satisfy the event ${\cal B}$ of  Lemma~\ref{l:shuffle} for $h$ and $c$. We show that then the conclusions of Lemma~\ref{l:color} hold deterministically.

Suppose that we are given a point in time~$t'$, a set of uncolored vertices~$U$ and a constant $4\xi < h\le \frac{1}{\alpha} + 3\xi$ which satisfy conditions $(i)$ and $(ii)$ from Lemma~\ref{l:applyboxgame}. Furthermore, assume that~$|U| \leq 2N\gamma(h)+1$ and assume that at time~$t'$ we have $|\calC_i| \ge (h-4\xi)\log_b np$ for every active color~$i$. We are interested in finding an arrangement of subsets $S(v) \subset A(v,\calC')$ such that $|S(v)| \leq q(h)$ holds for all vertices $v \in U$ and such that every color $i$ appears in at most $q(h)$ sets $A(v,\calC') \setminus S(v)$. 


\begin{algorithm}
\caption{{\sc ColorArranging}}
\begin{algorithmic}

\FOR{\textbf{all} $v \in U$}
\STATE $S(v) \leftarrow \emptyset$ 
\ENDFOR
\FOR{$i = 1$ to $k$}
\STATE $U(i) \leftarrow \{v \in U: i \in A(v,\calC')\}$
\IF{$|U(i)| > q(h)$}
\STATE $U'(i) \leftarrow \{$the $q(h)$ vertices $v \in U(i)$ with largest $|S(v)|\}$
\FOR{\textbf{all} $v \in U(i) \setminus U'(i)$}
\STATE put $i$ into $S(v)$
\ENDFOR
\ENDIF
\ENDFOR
\end{algorithmic}
\end{algorithm}
We use the algorithm {\sc ColorArranging} to find the sets $S(v)$. Note that this algorithm is essentially a greedy algorithm. We start with sets $S(v)$ that are empty for every vertex. Then we consider the colors one by one in an arbitrary order. Let $U(i)$ denote the set of all vertices in $U$ where color $i$ is available at time $t'$. We choose a subset $U'(i)\subseteq U(i)$ that consists of $q(h)$ vertices $v \in U(i)$ which (currently) have the largest set $S(v)$. (We break ties arbitrarily.) For all $v \in U(i)\setminus U'(i)$ we then add $i$ to the set $S(v)$.

At the end of the algorithm, every color $i$ appears by construction in at most $q(h)$ many sets $A(v,\calC') \setminus S(v)$. We need to show that the constructed arrangement also has the property that $s(v) := |S(v)| \leq q(h)$ for all $v \in U$. Assume for a contradiction that after termination of the algorithm there exists a vertex $w\in U$ such that  $s(w) > q(h)$. 
In the reminder of this section we show that the assumption that  the event ${\cal B}$ of  Lemma~\ref{l:shuffle} does not hold for $h$ and $c$ suffices to make this conclusion.
For $K\in\{0,\ldots,L\}$ let
\[W_K := \{v \in U: s(v) > (1-K  c)q(h)\}.\]
By assumption we have $|W_0| \ge 1$. Note also that all sets $W_K$ are subsets of $U$ and we thus have 
\begin{equation}\label{ineq:indhyp0}
|W_{K}| \leq |U| = O(\gamma(h)) \ll (np)^{h+\xi},
\end{equation}
where we use $(\ref{eq:boundonp})$, which holds by our lower bound on $p$. We will show by induction over $K$ that  
\begin{equation}\label{ineq:indhyp}
|W_{K}| \geq   c^{K} q(h)^{2K}(np)^{-12L\xi} \qquad\text{for all $K\in\{0,\ldots,L\}$}.
\end{equation}
Observe that this completes the proof of  Lemma~\ref{l:color}, as
\begin{align*}
 c^L q(h)^{2L} (np)^{-12L\xi} & \;=\;   c^L (q(h) (np)^{-6\xi})^{2L} \;\gg\;  \frac{  c^L (np)^{(1-h - 6\xi)2 L}}{\log^{8L} n}\;\gg\;
 (np)^{h+\xi},
%
\end{align*}
where the last step follows from $1-h-6\xi \ge \xi$ (cf.\ upper bound on $h$ and the definition of $\xi=\frac1{10}(1-\frac1\alpha)$) and our choice of $L$. That is, the validity of $(\ref{ineq:indhyp})$  for $K=L$ contradicts $(\ref{ineq:indhyp0})$, yielding the desired contradiction. 
 
%

It remains to prove that $(\ref{ineq:indhyp})$ holds. 
Let $K \le L$ and suppose (\ref{ineq:indhyp}) is true for $K-1$. By definition we have
$$s(v) \geq (1-(K-1) c)q(h)$$
for all $v \in W_{K-1}$. Thus, for every $v \in W_{K-1}$ there have to exist $c q(h)$ colors that were added to $S(v)$ at a time when $S(v)$ contained already at least  $(1-K  c)q(h)$ colors. We denote by $I(v)$ the set of exactly these colors, and put
\[I_{K-1} := \bigcup_{v \in W_{K-1}} I(v).\]
As we assumed that the event $\mathcal{B}$ of Lemma~\ref{l:shuffle} does not hold for $h$ and $c$, we deduce from property $(i)$ that for $S = W_{K-1}$ and $d=|I_{K-1}|$ we have
\begin{equation}\label{ineq:colsetsize}
|I_{K-1}| \geq  c q(h) \cdot |W_{K-1}| \cdot (np)^{-6\xi}
\end{equation}

All colors from the set $I_{K-1}$ have been added to some set $S(v)$ which contained already at least $(1-K  c)q(h)$ colors. By construction of the algorithm this means that for every $i \in I_{K-1}$ there were (at the time when color $i$ was processed) at least $q(h)$ other vertices $v'$ with $i \in A(v',\mathcal{C}')$ that also satisfied $s(v') > (1-K  c)q(h)$, hence all these vertices $v'$ lie in the set $W_K$. Applying Lemma~\ref{l:shuffle} with $S=W_K$ and $d =I_{K-1}$, we see that property $(ii)$ implies
\[
|W_K|  \stackrel{(ii)}\ge q(h) (np)^{-6\xi}|I_{K-1}| \stackrel{(\ref{ineq:colsetsize})}\ge  c q(h)^2 (np)^{-12\xi} |W_{K-1}| \stackrel{i.a.}\ge  c^{K} q(h)^{2K} (np)^{-12K\xi},
\]
and we conclude that (\ref{ineq:indhyp}) holds for $W_K$ as well.

\qed

\subsection{Proof of Theorem~\ref{thm:main}}\label{sec:proofofthm}

We define a set of constants $H := \{h_1, \ldots, h_J \}$
as follows: 
$h_1 := \frac{1}{2}-\xi$, where $\xi = \frac{1}{10}(1-\frac{1}{\alpha})$ as before. For $j>1$ we define $h_j := h_{j-1} + \xi$ and denote by
$J$ the smallest integer such that $h_J \geq \frac{1}{\alpha}+2\xi$. 

For the reminder of the proof we assume that all low-probability events in Section~3.1 do not occur for any pair $(h, c(h))$, where $h \in H$ and $c(h)$ as defined in $(\ref{eq:def:lc})$. We show that in this case Maker will win the game {deterministically}. Since $H$ is a finite set, this will thus conclude the proof of Theorem~\ref{thm:main}. 

Assume that Maker uses our proposed strategy, but at time $t$, he loses the coloring game. I.e., we assume that after $t-1$ rounds of the game, the two players obtain a partial coloring $\calC$ such that there is at least one uncolored vertex $v_0$ where no color is available. Our goal is to apply Lemma~\ref{l:applyboxgame} in order to show that Maker could not lose the game at time $t$.

Depending on the value of the time $t$ at which Maker supposedly loses the game we define a constant $h=h(\calC) \in H$ as follows. If $\calC$ contains at least $\lceil (np)^{1-4\xi}\rceil$ uncolored vertices, put
\[h(\calC) := \min\{h \in H: h \log_b np > \ell(\calC) + \xi\log_b np\}.\]
Note that $\ell(\calC) < (\frac{1}{\alpha} + \xi)\log_b np$ holds by Corollary~\ref{l:levelbound} and therefore $h(\calC)$ is well-defined.

If $\calC$ contains less than $\lceil (np)^{1-4\xi}\rceil$ uncolored vertices, we subsequently 
 remove the color assignment from the vertices which have been colored last in order to obtain a coloring $\overline{\calC}$ with exactly $\lceil (np)^{1-4\xi}\rceil$ uncolored vertices. We then let
$h(\calC) := h(\overline{\calC})$.

Let $j \in \{1, \ldots, J\}$ and suppose that $h(\calC) = h_j$. If $\calC$ contains at least $\lceil (np)^{1-4\xi} \rceil$ uncolored vertices, then it follows by Lemma~\ref{l:midlevel} that $|\calC_i| \le \ell(\calC) + \xi \log_b np$ holds for at least $\frac{\xi}{8}k$ color classes $\calC_i$. In the special case where $\calC$ contains less than $\lceil (np)^{1-4\xi} \rceil$ uncolored vertices, the same holds for the partial coloring $\overline{\calC}$ which can be extended to $\calC$ by coloring at most $\lceil (np)^{1-4\xi} \rceil$ additional vertices. Since $(np)^{1-4\xi} = o(k)$, we deduce that in this special case, at least $\frac{\xi}{10}k$ colors $i$ satisfy $|\calC_i| \le \ell(\calC) + \xi \log_b np$. 

We now define $t'$ as the last time before $t$ when Maker colored a vertex $v$ with at least $\beta(h(\calC))/2$ available colors in a move of his first type. By definition of such a move, we know that all uncolored vertices had at least $\beta(h(\calC))/2$ available colors at this time. That is, $t'$ and $h(\calC)$ satisfy conditions (i) and (ii) of Lemma~\ref{l:applyboxgame}. Moreover, by the definition of $t'$, we know that between $t'$ and $t$ Maker  always colored a vertex with less than $\beta(h(\calC))/2$ available colors in his moves of the first type. Lemma~\ref{l:sudakov} implies that even at time $t$ the number of vertices for which $a(v,\calC)$ is less than $\beta(h(\calC))/2$ is bounded by $\gamma(h(\calC))$. We thus deduce that Maker can have colored at most  $\gamma(h(\calC))$ vertices with a move of his first type between time $t'$ and time $t$. With that we have
\begin{equation}\label{eq:sizeofendgame}
t-t'+1 \leq 2N \cdot \gamma(h(\calC)) + 1 = o(n),
\end{equation}
where $N = N(\alpha)$ is the parameter from Maker's strategy. 
If we thus use
$U$ to denote the set of all vertices which have been colored in the period $[t',t-1]$, together with $v_0$, we have that 
\[|U| \leq 2N\gamma(h(\calC))+1\]
and $U$ thus satisfies the prerequisite of Lemma~\ref{l:color}. It remains to check  that the partial coloring $\calC'$  obtained after $t'-1$ rounds satisfies 
\begin{equation}\label{ineq:finalstep}
|C_i '| \ge (h(\calC)-4\xi)\log_b np
\end{equation}
for every active color $i \in \{1, \ldots k\}$, because then it follows from Lemma~\ref{l:color}  that condition~$(iii)$ of Lemma~\ref{l:applyboxgame} is fulfilled as well and Lemma~\ref{l:applyboxgame}  thus implies that Maker cannot lose the game at time $t$, since the induced box game is a Maker's win. 

Below we will prove (\ref{ineq:finalstep}) for the case $j>1$. If $j=1$ we 
show directly that Maker cannot lose the coloring game in this phase of the game. Indeed, suppose $j=1$ and suppose Maker has colored at time $t'$ a vertex $w$ such that $a(w,\mathcal{C}') \geq \beta(h_1)/2$. Because Maker has used there a move of first type, it follows $a(v,\mathcal{C}') \geq \beta(h_1)/2$ for every uncolored vertex at time $t'$. Note that
\[\beta(h_1) = \Theta((np)^{1/2+\xi} \log^{-1}{np}) \quad \text{and}\quad \gamma(h_1) = \Theta((np)^{1/2-\xi}p^{-1} \log{np}\log n).\]
By assumption $p$ satisfies $(\ref{eq:boundonp})$ and we deduce that $\beta(h_1) \gg \gamma(h_1)$. Hence, for $n$ large enough, (\ref{eq:sizeofendgame}) implies $\beta(h_1)/2 > t-t'+1$. As with every move of the game a vertex can lose at most one of its available colors this thus means that at time $t$ no vertex can have lost all available colors, contradicting our assumption that Maker lost the game at time $t$.

Therefore we can assume from now on that $j > 1$. For this we define  two additional point in time: 
$$\hat t = \min\{t, n-\lceil (np)^{1-4\xi} \rceil\} \quad\text{and}\quad t^{\ast} := \min\{t \in \mathds{N}: \ell(t) = (h(\calC) - 4\xi) \log_b np\}.$$ 
Observe that $\hat t\le t$. We will show below that $\hat t-t^{\ast} = \Theta(n)$. Thus, we have $t-t^{\ast} = \Theta(n)$ as well and $(\ref{eq:sizeofendgame})$ thus implies that $t^{\ast} \ll t'$. By definition, $\ell(t^{\ast}) \ge (h(\calC)-4\xi)\log_b np$. Clearly, the level of the game is increasing in time and (\ref{ineq:finalstep}) follows.

 It remains to show that $\hat t-t^{\ast} = \Theta(n)$.
Observe that Lemma~\ref{l:midlevel} implies that at time $t^{\ast}$ there exist at least $\frac{\xi}{8}k$ active colors classes of size at most $(h(\calC) - 3\xi) \log_b np$. 
On the other hand, note that $j > 1$ implies that 
$$\ell(\hat t) \geq (h(\calC) - 2\xi) \log_b np.$$
Hence all of these $\frac{\xi}{8}k$ colors either reached size at least $(h(\calC)-2\xi)\log_b np$ at time $\hat t$ or have been eliminated. From Lemma~\ref{l:eliminated} we know that by time $\hat t$ we have eliminated at most $\lceil (np)^{1-4\xi}\rceil=o(k)$ color classes with size less than $(h(\calC)-2\xi)\log_b np$. Thus, at least $\frac{\xi k}{8}-o(k)$ color classes increased by at least $\xi\log_b{np}$ in the period between $t^{\ast}$ and $\hat t$ and we deduce that
$$\hat t-t^{\ast} \geq \left(\frac{\xi k}{8}-o(k)\right) \cdot \xi\log_b np = \Theta(n),$$
which concludes the proof of Theorem~\ref{thm:main}.
\qed

\section{Open Questions}

In this paper we obtained the asymptotic value of the game chromatic number of $G_{n,p}$ for sufficiently dense graphs by showing that $\chi_g(G_{n,p}) = (2+o(1))\chi(G_{n,p})$ holds if $p \ge e^{-o(\log n)}$. However, there is no a priori reason why a similar statement should not be true for smaller values of $p$ as well. A key ingredient of our proof is 
Lemma~\ref{l:color} in which we show that the algorithm {\sc ColorArranging} yields sets $S(v)$ such that the induced box game is a Maker's win. For its correctness proof we need that $p \ge e^{-o(\log n)}$. Expanding Theorem~\ref{thm:main} to smaller edge probabilities $p$ thus seems to require different arguments. We leave this to future work. 


\vspace{0.3cm}
\noindent
\textbf{Acknowledgment.} The authors would like to thank Henning Thomas and József Balogh for valuable and motivating discussions on topics related to this paper.

\bibliographystyle{acm}
\bibliography{referencelist}

\end{document}